\documentclass[11pt]{amsart}
\usepackage{epsfig}
\usepackage{amssymb, amsmath}
\newtheorem{theorem}{Theorem}[section]

\newtheorem{proposition}[theorem]{Proposition}
\newtheorem{lemma}[theorem]{Lemma}
\newtheorem{claim}[theorem]{Claim}
\newtheorem{corollary}[theorem]{Corollary}
\newtheorem{remark}[theorem]{Remark}

\begin{document}

\title[Geodesically Tracking Quasi-Geodesic Paths for Coxeter Groups]
{Geodesically Tracking Quasi-Geodesic Paths for Coxeter Groups}

\author[Michael Mihalik]{Michael Mihalik}
\address{Department of Mathematics\\
        Vanderbilt University\\
        Nashville, TN 32340}
\email{michael.l.mihalik@vanderbilt.edu}
\email{steven.tschantz@vanderbilt.edu}

\author[Steve Tschantz]{Steve Tschantz}

\date{\today}


\begin{abstract}
If $\Lambda$ is the Cayley graph of a Gromov hyperbolic group, then it is a fundamental fact that quasi-geodesics in $\Lambda$ are tracked by geodesics. Let $(W,S)$ be a finitely generated Coxeter system and $\Lambda(W,S)$ the Cayley graph of $(W,S)$. For general Coxeter groups, not all quasi-geodesic rays in $\Lambda$ are tracked by geodesics. In this paper we classify the $\Lambda$-quasi-geodesic rays that are tracked by geodesics. As corollaries we show that if $W$ acts geometrically on a CAT(0) space $X$, then CAT(0) geodesics in $X$ are tracked by Cayley graph geodesics (where the Cayley graph is equivariantly placed in $X$) and for any $A\subset S$, the special subgroup $\langle A\rangle$ is quasi-convex in $X$. We also show that if $g$ is an element of infinite order for $(W,S)$ then the subgroup $\langle g\rangle$ is tracked by a Cayley geodesic in $\Lambda(W,S)$ (in analogy with the corresponding result for word hyperbolic groups).
\end{abstract}

\maketitle

\section{Introduction}

\medskip

Suppose $G$ is a group with finite generating set $A$, and $\Lambda(G,A)$ is the Cayley graph of $G$ with respect to $A$. If $G$ is word hyperbolic then any quasi-geodesic in $\Lambda$ is tracked by a geodesic (see \cite{Short}).  The corresponding result for CAT(0) groups is not true. Our main goal in this paper is to classify the quasi-geodesics in the Cayley graph of a finitely generated Coxeter system that are tracked by geodesics. We define a ``bracket number" for a Cayley path in terms of the wall crossings of the path and our main theorem is that a quasi-geodesic ray or line is tracked by a geodesic iff the bracket number of the ray (line) is bounded. Our principal corollary to this theorem states that if $(W,S)$ is a finitely generated Coxeter system, and $W$ acts geometrically on a CAT(0) space $X$, then the CAT(0) geodesics of $X$ are tracked by $(W,S)$-Cayley geodesics in $X$. If $X$ is the Davis complex for $(W,S)$ or even if $W$ acts as a reflection group on $X$, the proof of the corollary is straightforward. Unfortunately, the reflection group argument has no analogue when $W$ does not act as a reflection group on $X$. 
The principal corollary directly implies that if $A\subset S$ then the special subgroup $\langle A\rangle$ is quasi-convex in $X$. 

If a group $G$ acts geometrically on a CAT(0) space $X$ and one is interested in the asymptotic properties of $X$ it is a considerable advantage to know that CAT(0) geodesics in $X$ are tracked by Cayley geodesics. Clearly, the algebraic properties of $G$ are far more apparent in Cayley geodesics than in CAT(0) geodesics. This theme is highlighted in \cite{MRT} where local connectivity of boundaries of right angled Coxeter groups are analyzed. 

The work of B. Bowditch and G. Swarup (see \cite{Swarup}) imply that 1-ended word hyperbolic groups have locally connected boundary. One can easily see from our tracking results that any 1-ended hyperbolic Coxeter group has locally connected boundary. 

\section{Coxeter Preliminaries}

We use M. Davis' book \cite{Davis} as a general Coxeter group reference for this section. A Coxeter system is a pair $(W,S)$ where $S$ is a generating set for the group $W$ and $W$ has presentation 
$$\langle S: (s_is_j)^{m(i,j)}\hbox{ for all }s_i, s_j\in S\rangle$$  
where $m(i,j)\in \{1,2,\ldots , \infty\}$, $m(i,j)=1$ iff $i=j$ (so all generators are order 2) and $m(i,j)=m(j,i)$. If $m(i,j)=\infty$, the element $s_is_j$ is of infinite order (and the relation $(s_is_j)^\infty$ is left out of the presentation). 

A {\it reflection} in $W$ is a conjugate of an element of $S$. 
If $w\in W$ and $s\in S$ then the edge labeled $s$ in the Cayley graph $\Lambda(W,S)$ at the vertex $w$ is mapped to itself by the reflection $wsw^{-1}$, so that the vertices $w$ and $ws$ are interchanged. I.e. the edge is reflected across its midpoint. The set of edges in $\Lambda$ each fixed (set-wise) by a reflection is a {\it wall} of $\Lambda$. The walls of $\Lambda$ partition the edges of $\Lambda$ into disjoint sets. Notationally, we write a wall $Q$ as $[e]$ where $e$ is any edge of the wall $Q$ and we define $\bar Q$ to be the union of the edges of $Q$ in $\Lambda$. An edge $e$ (with say label $t\in S$) belongs to a wall $Q$ corresponding to the reflection $wsw^{-1}$ iff a vertex of $e$ is $wq$ where $qtq^{-1}=s$. The closure of the compliment of a wall in $\Lambda$ has exactly two components (which are interchanged by the reflection) called the {\it sides} of the wall. Two walls are {\it parallel} if all edges of one are on the same side of the other. If two walls are not parallel, then they {\it cross}. The following theorem due to B. Brink and R. Howlett (see theorem 2.8 of \cite{BrinkHowlett}) is a fundamental result concerning the wall structure of $\Lambda$.

\begin{theorem}\label{parallel} {\bf (Parallel Wall theorem)}
Suppose $(W,S)$ is a finitely generated Coxeter system and $\Lambda(W,S)$ the Cayley graph of $W$ with respect to $S$. For each positive integer $n$ there is a constant $P(n)$ such that the following holds: given a wall $Q$ and a point $p$ in $\Lambda$ such that the distance from $p$ to $\bar Q$ is at least $P(n)$, then there exist $n$ distinct pairwise parallel walls which separate $\bar Q$ from $p$. 
\end{theorem}

For a path $\beta$ in $\Lambda$ and vertex $t$ of $\beta$ let the {\it bracket number of $t$ in $\beta$} be the number of walls $Q$ such that there is an edge of $Q$ on either side of $t$ in $\beta$. Denote the bracket number of $t$ in $\beta$ as $B(t,\beta)$. If $\tau$ is a subpath of $\beta$ the {\it bracket number of $\tau$ in $\beta$} is the maximum of the numbers $B(t,\beta)$ for all vertices $t$ of $\tau$. Denote this number $B(\tau, \beta)$. Call $B(\beta)\equiv B(\beta, \beta)$ the {\it bracket number of $\beta$}.

\section{Wall computations}
If $\alpha$ is an edge path in the Cayley graph $\Lambda$ with consecutive vertices $a=v_0,v_1,\ldots, v_n =b$, then an $L$-{\it approximation to} $\alpha$ is an edge path in $\Lambda$ connecting $a$ and $b$ of the form $(\alpha_1,\ldots ,\alpha_n)$ where for all $i$, $\alpha _i$ is geodesic connecting $w_{i-1}$ to $w_i$ and $w_i$ is within $L$ of $v_i$. The points $w_i$ are called {\it approximation points}. 

\begin{lemma} Suppose $(W,S)$ is a finitely generated Coxeter system, $\alpha$ is an edge path in the Cayley graph $\Lambda (W,S)$ connecting $a$ and $b$, and $\beta$ is an $L$-approximation of $\alpha$. Then the bracket number $B(\beta)$ is bounded by a constant only depending on $B(\alpha)$, $L$ and constants independent of the choice of $\alpha$. 
\end{lemma}

\begin{proof} Let the consecutive vertices of $\alpha$ be $a=v_0, v_1,\ldots ,v_n=b$, the approximation vertices of $\beta$ be $a=w_0, w_1,\ldots ,w_m=b$ (so that $d_{\Lambda}(w_i,v_i)\leq L$ for all $i$) and $\beta _i$ be the geodesic subpath of $\beta$ connecting $w_{i-1}$ to $w_i$. Then $\beta=(\beta_1,\ldots ,\beta_m)$. 
If $x$ is a vertex of $\beta_i$ and $B(x,\beta)$ is ``large", then (as each edge belongs to exactly one wall) there is a wall $Q$ that brackets $x$ on $\beta$ that is ``far" from $x$ and hence far from $v_i$. 
Hence it suffices to bound the distance between $v_i$ and a wall $Q$ that brackets $x$ on $\beta$. The Parallel Wall theorem implies this distance is large iff there is a large set $\mathcal Q$ of (mutually parallel) walls that separate $\bar Q$ from $v_i$, so it suffices to bound 
the size of the set $\mathcal Q$ of walls that separate $\bar Q$ from $x$. Say $j<i<k$ such that $e_j$ and $e_k$ are edges of $\beta_j$ and $\beta _k$ respectively, and 
each of $e_j$, $e_k$ belongs to the wall $Q$. (See figure 1)

\medskip


\hspace{.5in}\includegraphics[scale=.9]{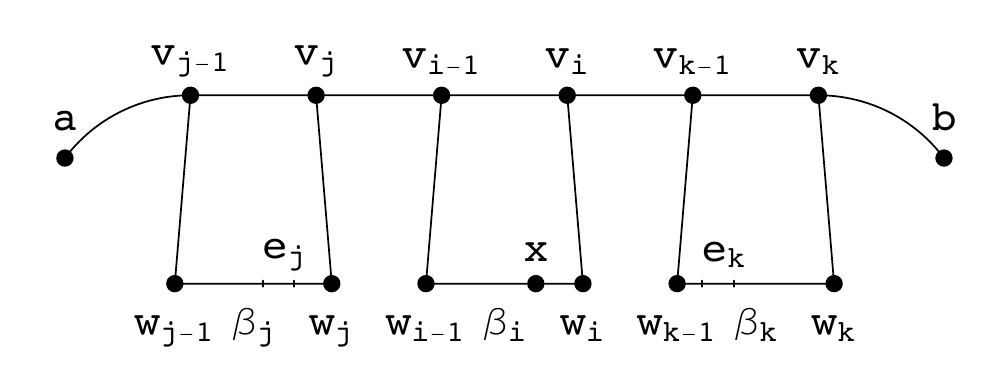}

\centerline{Figure 1}

\medskip


A path $\delta_j$, that begins at the end point of $e_j$ follows $\beta_j$ to $w_j$ and then travels geodesically from $w_j$ to $v_j$ has length $\leq 3L$. 
If $\alpha_{j,i}$ is the subpath of $\alpha$ from $v_j$ to $v_i$, then the path $(\delta _j,\alpha_{j,i})$ must cross each wall of $\mathcal Q$. Similarly define a path from $e_k$ to $v_i$ (which also crosses each wall of $\mathcal Q$). Then  at most $6L$ walls of $\mathcal Q$ do not bracket $v_i$ on $\alpha$. This bounds the size of $\mathcal Q$ by $6L+B(\alpha)$. 
\end{proof}

\begin{lemma} \label{reflect} Suppose $(W,S)$ is a Coxeter system and $\alpha=(e_1,\ldots, e_n)$ is a geodesic edge path connecting vertices $a$ and $b$ in $\Lambda(W,S)$ such that $\alpha$ does not cross the wall $Q$. If $e_0$ is an edge at $a$ and $e_{n+1}$ an edge at $b$ such that $e_0$ and $e_{n+1}$ belong to the wall $Q$ then each vertex of $\alpha$ is within $P(1)$ of $\bar Q$ (where $P$ is the function of theorem \ref{parallel}). In particular, if $v$ is a vertex of $\alpha$ and $v' $ the reflection of $v$ across $Q$ then $d(v,v')\leq 2P(1)+1$. 
\end{lemma}
\begin{proof} 
Otherwise, there is a wall $Q'$ separating a vertex $v$ of $\alpha$ from $Q$. Hence there is an edge of $\alpha$ between $a$ and $v$ that belongs to $Q'$ and an edge of $\alpha$ between $v$ and $b$ that belongs to $Q'$. This is impossible as $\alpha$ is geodesic.
\end{proof}

\begin{proposition} \label{rightside}
Suppose $(W,S)$ is a Coxeter system and $\alpha$ is an edge path of $\Lambda(W,S)$ connecting $a$ and $b$. Then there is an $L$-approximation $\beta$ to $\alpha$ such that each vertex of $\beta$ is on a geodesic connecting $a$ and $b$ and such that $L\leq (2P(1)+1)B(\alpha) $. 
\end{proposition}

\begin{proof}
Let the consecutive vertices of $\alpha$ be $a=v_0,\ldots ,v_n=b$. For $0<i<n$ we choose an approximation point $w_i$ for $v_i$ as follows. Let $\alpha_i$ be the geodesic from $a$ to $v_i$ and $\beta_i$ the geodesic from $v_i$ to $b$. Each wall of $ (\alpha _i,\beta_i)$ is crossed exactly once or twice. The number of  walls crossed twice by $ (\alpha _i,\beta_i)$ is 
$$N_i\equiv {1\over 2}(d(a,v_i)+d(v_i,b)-d(a,b))\leq B(\alpha)$$ 
Let $e$ be the last edge of $\alpha_i$ belonging to a wall which is crossed twice by $(\alpha _i,\beta_i)$ and $d$ the edge of $\beta_i$ in the same wall as $e$. (See figure 2.)

\medskip


\hspace{.6in}\includegraphics[scale=.9]{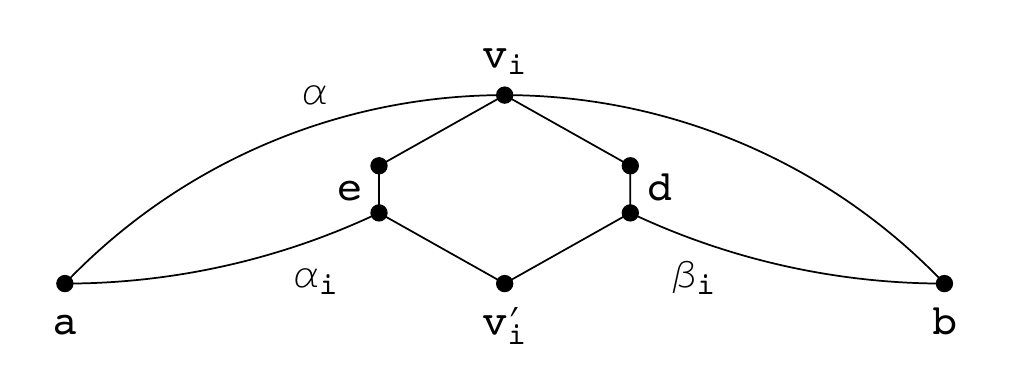}

\centerline{Figure 2}

\medskip



The segment of $(\alpha _i,\beta_i)$ between $e$ and $d$ is geodesic. Considering the reflection of this segment across the wall containing $e$ and $d$ (equivalently, delete $e$ and $d$ from $(\alpha _i,\beta_i)$). Then we see that $v_i'$, the reflection of $v_i$, is within $2P(1)+1$ of $v_i$ (lemma \ref{reflect}), and the distance from $v_i'$ to $a$ (respectively $b$) is less than that of $v_ i$ to $a$ (respectively $b$). Hence ${1\over 2}(d(a,v'_i)+d(v_i',b)-d(a,b))<N_i$ and a geodesic from $a$ to $v_i'$ followed by a geodesic from $v_i'$ to $b$ crosses at most $N_i-1$ walls twice.  Continuing as above at most $N_i(\leq B(\alpha))$ such reflections are needed to move $v_i$ to a point $w_i$ on a geodesic between $a$ and $b$, and so $d(w_i, v_i)\leq (2P(1)+1)B(\alpha)$. 

It remains to see that each vertex of a geodesic connecting $w_i$ and $w_{i+1}$ belongs to a geodesic connecting $a$ and $b$. Consider the edge path $(\delta_i, \beta_i, \gamma_i)$ where $\delta_i$ is a geodesic connecting $a$ to $w_i$, $\beta_i$ is a geodesic connecting $w_i$ to $w_{i+1}$ and 
$\gamma_i$ is a geodesic connecting $w_{i+1}$ to $b$. The paths $\delta_i$ and $\gamma_i$ only cross walls  crossed by some (equivalently any) geodesic connecting $a$ to $b$. If a vertex $v$ of 
$\beta_i$ is not on a geodesic connecting $a$ and $b$ then there is a wall $R$ separating $v$ from
 some (equivalently every) geodesic connecting $a$ and $b$. As $R$ separates $v$ from $a$, and 
 $\delta_i$ does not cross $R$, $\beta_i$ must cross $R$ between $w_i$ and $v$. 
 Similarly $\beta_i$ must cross $R$ between $v$ and $w_{i+1}$. This is impossible as $\beta_i$ is geodesic.
\end{proof}

If $\gamma$ is an edge path in $\Lambda$ connecting the vertices $a$ and $b$, then each wall separating $a$ and $b$ is crossed an odd number of times by $\gamma$ and each wall not separating $a$ and $b$ is crossed and even number of times by $\gamma$.  If two edges of $\gamma$ belong to the same wall then they may be ``deleted" to obtain another path from $a$ to $b$ (i.e. if edges $e$ and $d$ of $\gamma$ belong to the same wall $Q$, and $\tau$ is the segment of $\gamma$ between $e$ and $d$, then $(e,\tau, d)$ can be replaced in $\gamma$ by $\tau'$, where $\tau'$ is the reflection of $\tau$ across $Q$, to obtain a shorter path connecting $a$ and $b$). If $\alpha$ is a geodesic connecting $a$ and $b$ then the walls separating $a$ and $b$ are the walls determined by the edges of $\alpha$, so the walls separating $a$ and $b$ are in 1-1 correspondence with the edges of some (any) geodesic connecting $a$ and $b$. The following observations are straightforward.

\begin{lemma} \label{easy}
Suppose $\beta$ is an edge path in $\Lambda$ connecting the vertices $a$ and $b$ such that each vertex of $\beta$ is on a geodesic connecting $a$ and $b$. Then 

i) each edge of $\beta$ belongs to a wall that separates $a$ from $b$, 

ii) each wall crossed by $\beta$ is crossed an odd number of times, and 

iii) if $c$ and $d$ are vertices of $\beta$ then any wall separating $c$ and $d$ also separates $a$ and $b$. 
\end{lemma}

The next result is a slightly more sophisticated version of lemma \ref{reflect}.

\begin{lemma} \label{near}
Suppose $\alpha$ is a geodesic edge path in $\Lambda$ connecting the vertices $a$ and $b$, $v$ is a vertex of $\alpha$, and $a$ and $b$ are each within distance $A$ of $\bar Q$ for some wall $Q$. Then $v$ is within distance $2A(2P(1)+1)+P(1)$ of $\bar Q$. 
\end{lemma}

\begin{proof} Let $a'$ (respectively $b'$) be a vertex of $\bar Q$ within $A$ of $a$ (respectively $b$) and on the same side of $Q$ as is $a$ (respectively $b$). Let $\beta$ (respectively $\gamma$) be a geodesic from $a'$ to $a$ (respectively $b$ to $b'$). 

\medskip

\noindent {\bf Case 1.} The geodesic $\alpha$ does not cross $Q$.

\noindent In this case the path $\delta _0\equiv (\beta, \alpha, \gamma)$ does not cross $Q$. Since $\vert \beta\vert\leq A$ and $\vert \gamma\vert\leq A$, a sequence of at most $2A$ deletions (the first in the path $\delta_0$) determines a geodesic connecting $a'$ to $b'$ which does not cross $Q$. 

\medskip

{\bf $\ast )$} Each deletion is taken so that if $e$ and $d$  are the deleted edges, then the subpath determined by $e$ (or $d$) along with the subpath between $e$ and $d$ is geodesic. 

\medskip

If $e_1$ and $d_1$ are the first such deletion edges (so $e_1$ and $d_1$ are edges of $\delta _0$) then let $\delta_1$ be obtained from $\delta _0$ by deleting $e_1$ and $d_1$. If $v$ is not between $e_1$ and $d_1$ then $v$ is a vertex of $\delta_0$. If $v$ is between $e_1$ and $d_1$, then $v_1$, the reflection of $v$ across the wall $[e_1]=[d_1]$, is within $2P(1)+1$ of $v$, by lemma \ref{reflect}. (Note that the hypotheses of lemma \ref{reflect} are satisfied since we require condition $\ast $.) In any case $\delta _1$ contains a vertex $v_1$ within $2P(1)+1$ of $v$. If $e_2$ and $d_2$ are deleting edges of $\delta_1$ (satisfying $\ast $), then let $\delta _2$ by obtained from $\delta_1$ by deleting $e_2$ and $d_2$. Lemma \ref{reflect} implies $\delta_2$ contains a vertex $v_2$ within $2P(1)+1$ of $v_1$ and so within $2(2P(1)+1)$ of $v$. Inductively, after $K\leq 2A$ deletions, we obtain a geodesic $\delta_K$ connecting $a'$ and $b'$, and $\delta_K$ contains a vertex $v_K$ within $K(2P(1)+1)$ of $v$.  Note that $\delta_k$ does not cross $Q$. By lemma \ref{reflect}, $v_K$ is within $P(1)$ of $\bar Q$ so that $v$ is within $2A(2P(1)+1)+P(1)$ of $\bar Q$.  This completes case 1.

\medskip

\noindent {\bf Case 2.} Suppose $\alpha$ crosses $Q$.

\noindent Say the edge $e$ of $\alpha$ between $v$ and $b$ belongs to $Q$. Repeat the case 1 argument with $\delta_0$ replaced by $(\beta, \alpha')$, where $\alpha'$ is the subsegment of $\alpha$ from $a$ to the initial point of $e$. Similarly if $e\in Q$ is an edge of $\alpha$ between $a$ and $v$. Note that in both case 2 scenarios, at most $A$ deletions are required to straighten to a geodesic, so the bound is reduced to $A(2(P(1)+1)+P(1)$. 
\end{proof}

\section{Tracking Quasi-geodesics}

We are interested in quasi-geodesic edge paths in $\Lambda$. An {\it edge path} in $\Lambda$ is a continuous map $\beta:[0,n]\to \Lambda$ such that $n\in \mathbb Z^+$ and for each non-negative integer $k$, $\beta$ maps the interval $[k,k+1]$ isometrically to an edge of $\Lambda$. Similarly if $\beta:[0,\infty)\to \Lambda$, then $\beta$ is called a {\it ray} and, if $\beta:(-\infty, \infty)\to \Lambda$ then $\beta$ is called a {\it line}. An edge path $\beta$ is a $(\lambda,\epsilon)$-{\it quasi-geodesic} if for each pair of integers $s$ and $t$, $\vert s-t\vert\leq \lambda d(\beta(s),\beta(t))+\epsilon$. If $\alpha$ and $\beta$ are edge paths, then $\beta$ is $K${-\it tracked} by $\alpha$ if each vertex of $\beta$ is within $ K$ of a vertex of $\alpha$.

\begin{lemma} \label{doubletrack}
For $i\in \{1,2\}$ suppose $\beta_i$ is a $(\lambda_i,\epsilon_i)$-quasi-geodesic edge path in $\Lambda$, $\beta _1$ is $K$-tracked by $\beta _2$ and $\beta_1(0)$ is within $K$ of $\beta_2(0)$. Assume both $\beta_1$ and $\beta_2$ are bi-infinite, or both are rays, or both are finite length and the terminal points of $\beta_1$ and $\beta_2$ are within $K$ of one another. Then $\beta_2$ is $(\lambda_2(2K+1)+\epsilon_2+K)$-tracked by $\beta_1$. 
\end{lemma}
\begin{proof} 
Since each vertex of $\beta_1$ is within $K$ of a vertex of $\beta_2$, we may
define an integer function $a$ such that for each integer $i$ (in the domain of $\beta_1$), $\beta_1(i)$ is within $K$ of $\beta_2(a(i))$. We take $a(0)=0$ and if $\beta_i$ has $n_i$ edges then $a(n_1)=n_2$.

The first two inequalities follow from the definitions and the third follows from the first two.

$${\bf 1)} \ \ \ \ \ {\vert a(m+i)-a(m)\vert-\epsilon_2\over \lambda_2}-2K\leq d(\beta_2(a(m+i)),\beta_2(a(m)))-2K\leq$$ $$ d(\beta_1(m+i),\beta_1(m))\leq d(\beta_2(a(m+i)),\beta_2(a(m)))+2K\leq \vert a(m+i)-a(m)\vert+2K$$

$$ {\bf 2)} \ \ \ \ \ \ \ \ {i-\epsilon _1\over \lambda_1}\leq d(\beta_1(m+i),\beta_1(m))\leq i \ \ \ \ \ \ \ \ \ \ \ \ \ \ \ \ \ \ \ \ \ \ \ \ \ \ \ \ \ \ \ \ \ \ \ \ \ \ \ \ \ \ \ $$

$${\bf 3)}\ \ \ \ \ \ \ \ \ \  {i-\epsilon _1\over \lambda_1}-2K\leq \vert a(m+i)-a(m)\vert \leq \ \ \ \ \ \ \ \ \ \ \ \ \ \ \ \ \ \ \ \ \ \ \ \ \ \ \ \ \ \ \ \ \ \ \ \ \ $$ $$\lambda_2 (d(\beta_1(m+i),\beta_1(m))+2K)+\epsilon _2\leq (i+2K)\lambda_2+\epsilon _2$$

The inequality  $\vert a(i+1)-a(i) \vert \leq \lambda_2(2K+1)+\epsilon _2$ implies if $k$ is between $a(i)$ and $a(i+1)$ for some $i$ then $\beta_2(k)$ is within $\lambda_2(2K+1)+\epsilon _2+K$ of $\beta_1(i)$. In the case $\beta_1$ and $\beta _2$ are finite, the condition that terminal points are within $K$ of one another (so that $a(n_1)=n_2)$ implies that every integer in the domain of $\beta_2$ is between $a(i)$ and $a(i+1)$ for some $i$ and this case is finished. 
If $\beta_1$ and $\beta_2$ are rays then $a(i)$ is non-negative and equation 3) (with $m=0$) implies $a(i)$ is arbitrarily large for large $i$ 
and again every integer in the domain of $\beta_2$ is between $a(i)$ and $a(i+1)$ for some $i$. 
If $\beta_1$ and $\beta_2$ are bi-infinite, then the $a(i)$ may be positive or negative and (again by 3)) for large $\vert i\vert$, $\vert a_i\vert$ is large, and $lim_{i\to +\infty} a(i)=\pm \infty$ and $lim_{i\to -\infty} a(i)=\pm \infty$. It remains to see $lim_{i\to +\infty} a(i)\ne lim_{i\to -\infty}a(i)$. Equality is impossible, since otherwise, for every large positive integer $i$, $a(-i)$ would be between $a(j)$ and $a(j+1)$ for some (depending on $i$) large positive integer $j$. But equation 3) implies $a(j)$ and $a(j+1)$ are relatively close and $a(-i)$ and $a(j)$ are far apart.
\end{proof}

\begin{proposition}\label{BBN}
Suppose $\beta$ is a quasi-geodesic edge path ray in $\Lambda$ and $\beta$ is tracked by a geodesic, then $\beta$ has bounded bracket number.
\end{proposition}
\begin{proof} 
Assume that $\beta$ is a $(\lambda,\epsilon)$-quasi-geodesic. 
Suppose $\alpha$ is a geodesic such that each vertex of $\beta$ is within $L$ of a vertex of $\alpha$. For each integer $n\geq 0$, choose an integer $a(n)$ such that $d(\beta(n),\alpha(a(n)))\leq L$. We assume that $a(0)=0$.

The next two equations follow from the definitions and the third follows from the first two.  
$$ a(n)-2L\leq d(\beta(n),\beta(0))\leq a(n)+2L$$
$$ {n-\epsilon \over \lambda} \leq d(\beta(n),\beta(0) )\leq n$$
$$ {n-\epsilon \over \lambda} -2L\leq a(n)\leq n+2L$$

\begin{claim} \label{CL}
Suppose $K$ is an integer larger than $\lambda(4L +1)+\epsilon$. Then for any integer $n$, $a(n+K)>a(n)$. 
\end{claim}

\begin{proof} Note that if $m\geq \lambda(n +4L)+\epsilon$ then $a(m)>n+2L>a(n)$.
So if $K>\lambda(4L+1)+\epsilon $, and $a(n+K)\leq a(n)$, then there is a last integer $K_1>\lambda(4L+1)+\epsilon $ such that $a(n+K_1)\leq a(n)$. Then (see figure 3) 
$$a(n+K_1+1)>a(n)\geq a(n+K_1)$$

\medskip


\hspace{1.in}\includegraphics[scale=.9]{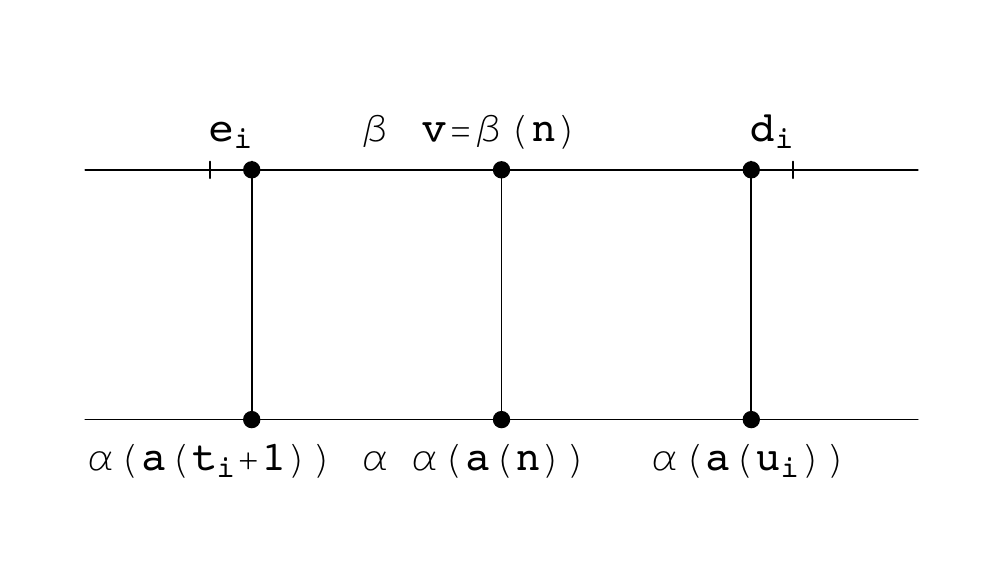}

\centerline{Figure 3}

\medskip

\medskip


Since $d(\beta(n+K_1),\beta(n+K_1+1))= 1$ for all $n$, and $d(\beta (i), \alpha(a(i)))\leq L$ for all $i$, we have 
$$d(\alpha(a(n+K_1)), \alpha(a(n+K_1+1)))\leq 2L+1$$
But as $\alpha (a(n))$ is between $\alpha (a(n+K_1))$ and $\alpha (a(n+K_1+1))$ on the geodesic $\alpha$, 
$$d(\alpha(a(n)), \alpha (a(n+K_1+1)))\leq 2L+1$$
Then $d(\beta (n),\beta(n+K_1+1))\leq 4L+1$. But 
$$d(\beta(n),\beta(n+K_1+1))\geq {1\over \lambda}(K_1+1-\epsilon)> 4L+1$$ 
the desired contradiction (so the claim is proved). 
\end{proof}

Now suppose $v\equiv \beta(n)$ is a vertex of $\beta$ with bracket number at least $2\lambda (4L+1)+2\epsilon +K$. Then (by the pigeon hole principal)  there are $K$ distinct walls, $Q_1,\ldots ,Q_K$ such that for each $i\in \{1,\ldots ,K\}$, there is an edge $e_i$ of $\beta$ preceding $v$ and an edge $d_i$ of $\beta$ following $v$ such that $e_i$ and $d_i$ belong to the wall $Q_i$, the subpath of $\beta$ between $e_i$ and $d_i$ does not cross $Q_i$, $e_i$ is not one of the $\lambda(4L +1)+\epsilon$ edges of $\beta$ immediately preceding $v$ and $d_i$ is not one of the $\lambda(4L+1)+\epsilon $ edges of $\beta$ immediately following $v$. I.e. $e_i=\beta ([t_i,t_{i}+1])$ where $t_{i}+1\leq n-\lambda(4L +1)-\epsilon$ and $d_i=\beta ([u_i,u_{i}+1])$ where $u_i\geq n+\lambda(4L +1)+\epsilon$.  (See figure 4.)

\medskip


\hspace{.5in}\includegraphics[scale=.9]{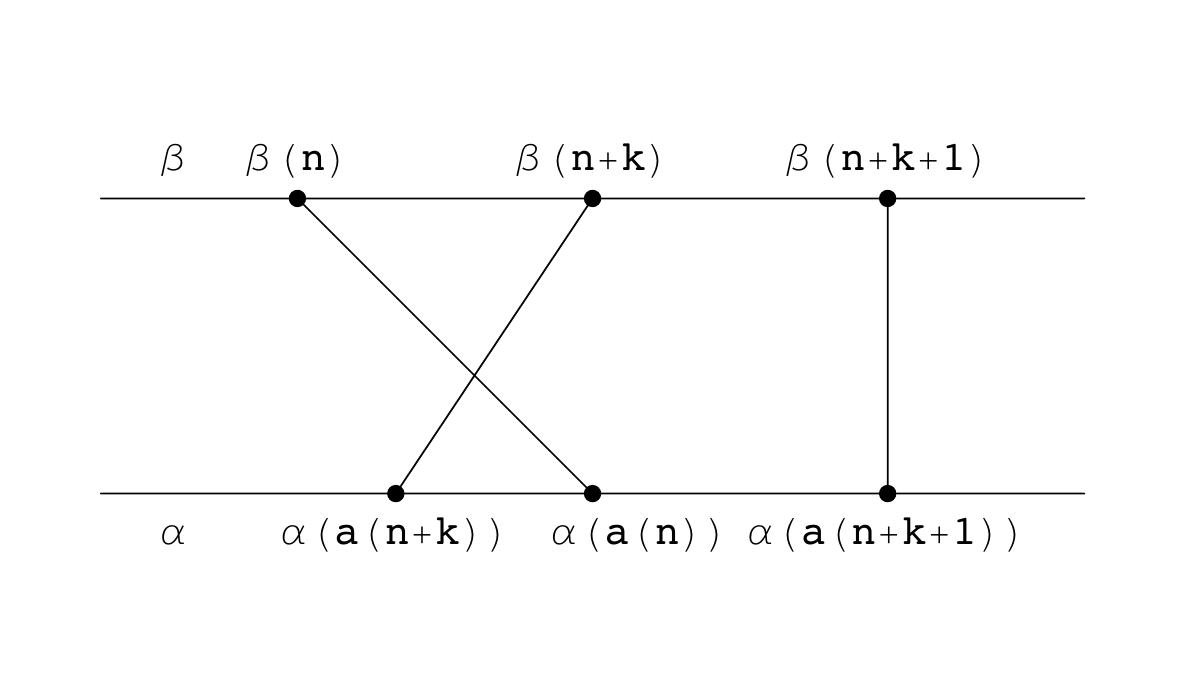}

\centerline{Figure 4}

\medskip

\medskip


By  claim \ref{CL}, $a(t_i+1)<a(n)<a(u_i)$. Hence, by lemma \ref{near}, $\alpha(a(n))$ is within $2L(2P(1)+1)+P(1)$ of the wall $Q_i$. For $x$ a vertex of $\Lambda$, let $C(k)$ be the number of  distinct walls that pass within $k$ of $x$. Note that $C$ is independent of vertex in $\Lambda$. Hence $K\leq C(2L(2P(1)+1)+P(1))$, bounding the bracket number of a vertex of $\beta$. 
\end{proof}

\section{Proof of Main Theorem}

In order to prove the main theorem, we need two results, one due to B. Brink and R. Howlett \cite{BrinkHowlett}, and a second, due to R. P. Dilworth \cite{Dilworth}.
 
\begin{theorem} \label{antichain}
(Brink-Howlett) Suppose $(W,S)$ is a finitely generated Coxeter system, and $\Lambda(W,S)$ is the Cayley graph of $W$ with respect to $S$. There is a bound $F_{(W,S)}$ on the number of mutually crossing walls of $\Lambda$.
\end{theorem}
  
Dilworth's theorem requires several definitions. If $A$ is a partially ordered set (a set with reflexive, antisymmetric and transitive binary relation $\leq$ on $A$), then any two elements $x$ and $y$ are {\it comparable} if either $x\leq y$ or $y\leq x$. Otherwise they are in {\it incomparable}. A subset $C$ of $A$ is a {\it chain} when every pair of points in $C$ is a comparable pair. A subset $B$ of $A$ is called an {\it anitchain} when every pair of points in $B$ is an incomparable pair. The number of points in a maximal antichain is called the {\it width} of $A$. 

\begin{theorem}
(Dilworth) If $A$ is a partially ordered set of width $w$, then $A$ can be partitioned into $w$ chains.
\end{theorem}

Suppose $x$ and $y$ are vertices of $\Lambda(W,S)$ and ${\mathcal W}_{(x,y)}$ is the set of walls that separate $x$ and $y$. We partially order ${\mathcal W}_{(x,y)}$ by saying $P\leq Q$ if either $P=Q$, or $P$ and $Q$ are parallel and $P$ separates $x$ from $Q$. Note that $P$ and $Q$ are parallel walls of ${\mathcal W}_{(x,y)}$, iff they are comparable. Hence $P$ and $Q$ are incomparable iff they cross. By proposition \ref{antichain}, the width of ${\mathcal W}_{(x,y)}$ is $F_{(W,S)}$. Applying Dilworth's theorem we have:

\begin{proposition}\label{partition}
Suppose $(W,S)$ is a finitely generated Coxeter system, and $\Lambda(W,S)$ is the Cayley graph of $W$ with respect to $S$. For any vertices $x$ and $y$ of $\Lambda$ the walls separating $x$ and $y$ can be partitioned into at most $F_{(W,S)}$ chains (where any two walls in the same chain are parallel).
\end{proposition}

Say a path is {\it geodesic with respect to a set of walls} if the path crosses each wall of the set either $0$ or $1$ times. The following lemma is clear.

\begin{lemma} \label{clear}
Suppose $\alpha$ is an edge path in $\Lambda$ and $\alpha$ is geodesic with respect to the set of parallel walls $\mathcal Q$. If a subpath of $\alpha$ is replaced by a geodesic edge path, then the resulting edge path is geodesic with respect to $\mathcal Q$.
\end{lemma}

\begin{theorem} \label{trackqg}
Suppose $(W,S)$ is a finitely generated Coxeter system, $\alpha$ is a $(\lambda,\epsilon)$-quasi-geodesic edge path from $a$ to $b$ in the Cayley graph $\Lambda (W,S)$. Then there is an integer $K$, depending only on $\Lambda$, $\lambda$, $\epsilon$ and the bracket number $B$ of $\alpha$, and a $\Lambda$-geodesic $\beta$ connecting $a$ and $b$ such that $\alpha$ is $K$-tracked by $\beta$. 
\end{theorem}

\begin{proof} The proof is a double induction argument. By proposition \ref{partition}, the walls separating $a$ and $b$ can be partitioned into at most $F$ sets ${\mathcal Q}_1,\ldots ,{\mathcal Q}_A$, where two walls in the same set are parallel. The ``outside" induction is on the number $A(\leq F)$ of sets of walls separating $a$ and $b$. The fact that $A$ is bounded by $F$ is critical to the argument that follows. Note that if $A=1$ then all walls separating $a$ and $b$ are parallel. In this case, the walls separating $a$ and $b$ are ordered as $Q_1,\ldots, Q_m$ where for $i<j<k$, $Q_j$ separates $Q_i$ from $Q_k$. Hence, there is a unique, geodesic edge path $\beta$ connecting $a$ and $b$, and $\beta$ crosses $Q_1$, then $Q_2$, $\ldots$. By proposition \ref{rightside}, the path $\alpha$ is approximated by a path $\alpha'$, such that each vertex of $\alpha'$ is on a geodesic connecting $a$ and $b$. The path $\alpha'$ only crosses the walls separating $a$ and $b$ (see lemma \ref{easy})  and, in this case, is geodesic, modulo backtracking. Eliminating backtracking on $\alpha'$ produces $\beta$.  Each vertex of $\alpha'$ is a vertex of $\beta$ and the basis case is complete.

Assume the statement of the theorem is true if $A$, the number of sets of walls separating $a$ and $b$, is less than or equal to $M-1$. 
Suppose there are $M$ sets of walls (${\mathcal Q}_1,\ldots ,{\mathcal Q}_M$) separating $a$ and $b$. 
By proposition \ref{rightside} we may assume every vertex of $\alpha$ is on a geodesic connecting $a$ and $b$, so that $\alpha$ only crosses walls separating $a$ and $b$ and $\alpha$ 
crosses each such wall an odd number of times.
The second induction is on $N(\leq M)$, the number of sets of walls, ${\mathcal Q}_i$, such that $\alpha$ is not geodesic with respect to ${\mathcal Q}_i$. If $N=0$, then $\alpha$ is geodesic. 
Assume the statement of the theorem is true for $N=K-1$ (when the number of sets of walls separating $a$ and $b$ is $\leq M$). 
Assume the ${\mathcal Q}_i$ are arranged so that $\alpha$ is geodesic with respect to ${\mathcal Q}_i$ for $K+1\leq i\leq M$. 
Write $\alpha$ as $(e_1,\ldots ,e_n)$ with consecutive vertices $a\equiv a_1, \ldots, a_n\equiv b$. Let $i$ be the first integer such that $e_i$ 
is an edge of a wall of ${\mathcal Q}_K$ and for some $j>i$, $e_j$ and $e_i$ are in the same wall $Q$. Now assume $j$ is the largest integer such that $e_j\in Q$. 
Since $\alpha$ crosses $Q$ an odd number of times, the path $\alpha _{i,j}\equiv (e_i,\ldots ,e_{j-1})$ (from $a_i$ to $a_j$) crosses $Q$ an even number of times. 
A geodesic $\beta_{i,j}$ connecting $a_i$ to $a_j$ does not cross $Q$. Since all walls of ${\mathcal Q}_K$ are parallel to one another, $\beta_{i,j}$ does not cross a wall of ${\mathcal Q}_K$. Hence $a_i$ and $a_j$ are not separated by a wall of ${\mathcal Q}_K$. 
By proposition \ref{rightside}, $\alpha_{i,j}$ is close to $\alpha _{i,j}'$ a quasi-geodesic edge path connecting $a_i$ to $a_{j}$, such that each vertex of $\alpha_{i,j}'$ is on a geodesic 
connecting $a_i$ to $a_j$. By lemma \ref{easy}, each wall separating $a_i$ and $a_j$ also separates $a$ and $b$, and the number of sets of walls 
separating $a_i$ and $a_j$ is less than $M$. 
By (outside) induction, there is a geodesic $\beta _{i,j}$ connecting $a_i$ and $a_j$ which tracks $\alpha_{i,j}'$ and therefore tracks $\alpha_{i,j}$. 
Replace $\alpha_{ i,j}$ by $\beta _{i,j}$. The resulting path, $\alpha_1$ crosses $Q$ exactly once at $e_j$. The walls of ${\mathcal Q}_K$ are ordered as $Q_1, Q_2,\ldots$ so that if $i<j$, then $Q_i$ separates $a$ from $Q_j$, and $Q_j$ separates $Q_i$ from $b$.  A wall of ${\mathcal Q}_K$ preceding $Q$ in this ordering is not crossed by $\alpha_1$ after $e_j$. Hence if ${\mathcal Q}\subset {\mathcal Q}_K$ is the set of walls of ${\mathcal Q}_K$ preceding $Q$ and including $Q$, then $\alpha_1$ is geodesic with respect to ${\mathcal Q}$ and (by lemma \ref{clear}), $\alpha_1$ is geodesic with respect to each set ${\mathcal Q}_i$ for $i>K$. Suppose $e_k$ is the first edge of $\alpha_1$ such that $e_k$ is an edge of a wall $Q$ of ${\mathcal Q}_K$, and for some $l>k$, $e_l\in Q$. Then $e_k$ follows $e_j$ on $\alpha_1$, and if we assume $e_l$ is the last edge of $\alpha_1$ in $Q$, then as above $(e_k,\ldots ,e_{l-1})$ can be replaced by a geodesic close to $(e_k,\ldots ,e_{l-1})$. Continuing, the resulting path is geodesic with respect to ${\mathcal Q}_K$ and by induction, the theorem follows. Note that the bound $F$ for $\Lambda$ (on the number of sets of parallel walls are necessary to partition the set of walls separating two points $a$ and $b$ of $\Lambda$), limits the total number of times the induction steps are carried out to arrive at a geodesic.  
\end{proof}  

\section{Consequences of the Main Theorem} 

\begin{corollary}\label{localfinite}
Suppose $(W,S)$ is a finitely generated Coxeter system, and $\Lambda(W,S)$ is the Cayley graph of $W$ with respect to $S$. Any infinite or bi-infinite $(\lambda,\epsilon)$-quasi-geodesic edge path $\alpha$ with bounded bracket number $B$ is $K'$-tracked by an edge path geodesic where $K'$ is a constant only depending on $\lambda$, $\epsilon$, $B$ and $S$. 
\end{corollary}
\begin{proof}
The proof is a standard local finiteness argument in both the infinite and bi-infinite case. We give the bi-infinite case. Write $\alpha$ as the edge path $(\ldots, e_{-1}, e_0,e_1,\ldots )$ in $\Lambda$. Let $v_i$ be the initial point of $e_i$. By theorem \ref{trackqg}, there is a $\Lambda$-geodesic $\beta_n$ which $K$-tracks $\alpha_n\equiv (e_{-n},\dots ,e_n)$. Note that every vertex of $\beta_n$ is within $2K$ of a vertex of $\alpha$. For each positive integer $n$, some vertex $x_n$ of $\alpha_n$ is within $K$ of $v_0$. Hence there is an infinite number of $x_n$ that are equal. Of this infinite subcollection of $x_n$, infinitely many have the same pair of edges one preceding and one following $x_n$ on $\beta_n$, of this infinite collection of $x_n$ there is an infinite subcollection that have the same four edges - the two preceding and the two following $x_n$ being exactly the same. Continuing, we have a bi-infinite geodesic $\beta$ and each vertex of $\beta$ is within $2K$ of a vertex of $\alpha$. As $\alpha$ is a $(\lambda ,\epsilon)$-quasi-geodesic, lemma \ref {doubletrack} implies each point of $\alpha$ is within $\lambda(4K+1)+\epsilon+2K$ of $\beta$. 
\end{proof}

The next result follows directly from proposition \ref{BBN} and corollary \ref{localfinite}.

\begin{corollary}
Suppose $(W,S)$ is a finitely generated Coxeter system, and $\Lambda(W,S)$ is the Cayley graph of $W$ with respect to $S$. Then a quasi-geodesic edge path ray in $\Lambda$ is tracked by a geodesic iff it has bounded bracket number.
\end{corollary}
 
A metric space $(X,d)$ is a called a {\it geodesic metric
space} if every pair of points are joined by a geodesic.
It is {\it proper} if for any $x\in X$, the ball of radius $r$ about $X$ is compact for all positive numbers $r$. A group $W$ acts {\it geometrically} on a space if the action is properly discontinuous, co-compact and by isometries. 

Let $(X,d)$ be a proper complete geodesic
metric space.
If $\vartriangle abc$ is a geodesic triangle in $X$, then we consider
$\vartriangle\overline a\overline b\overline c$ in $\mathbb E^2$, a
triangle with the same side lengths, and call this a {\it comparison
triangle}.  Then we say $X$ satisfies the $CAT(0)$ {\it inequality}
if given $\vartriangle abc$ in $X$, then for any comparison triangle and
any two points $p,q$ on $\vartriangle abc$, the corresponding points
$\overline p,\overline q$ on the comparison triangle satisfy
$$d(p,q)\leq d(\overline p,\overline q)$$

If $(X,d)$ is a $CAT(0)$ space, then the following basic properties
hold:
\begin{enumerate}
\item  The distance function $d\colon X\times X\to\mathbb R$ is convex.
\item  $X$ has unique geodesic segments between points.
\item  $X$ is contractible.
\end{enumerate}
For details, see \cite{BridsonHaefliger}.

Suppose $(W,S)$ is a finitely generated Coxeter system, $\Lambda(W,S)$ is the Cayley graph of $W$ with respect to $S$ and $W$ acts geometrically on a CAT(0) space $X$. Define $\Lambda_x\subset X$ to have as vertices, the orbit $Wx$, and CAT(0) geodesic edge connecting $w_1x$ and $w_2x$ (for $w_i\in W$) when there is $s\in S$ such that $w_1s=w_2$.  There is a proper $W$-equivariant map $P_x:\Lambda \to \Lambda_x$ so that $P_x$ maps the identity vertex of $\Lambda$ to $x$.

Intuitively, the next result says that when a Coxeter group acts geometrically on a CAT(0) space, CAT(0) geodesics are tracked by Cayley graph geodesics. This result generalizes the right angled version of the same result in \cite{MRT}.

\begin{corollary} \label{CAT(0)}
Suppose $(W,S)$ is a finitely generated Coxeter system, and $\Lambda(W,S)$ is the Cayley graph of $W$ with respect to $S$ and $W$ acts geometrically on the proper CAT(0) space $X$. If $x$ is a point of $X$ not fixed by any element of $W$, and $\Lambda_x$ is the copy of $\Lambda$ at $x$, then any CAT(0) geodesic ray in $X$ is tracked by a Cayley graph geodesic in $\Lambda_x$. 
\end{corollary}
\begin{proof}
For a given CAT(0) geodesic $\alpha$ we find a Cayley graph geodesic $\beta$ such that $P_x(\beta)$ tracks $\alpha$. It suffices to find $\lambda$, $\epsilon$, $K$ and $B$ such that any (finite) CAT(0) geodesic $\alpha$ is $K$-tracked by a Cayley $(\lambda ,\epsilon)$-quasi-geodesic with bracket number $\leq B$. Since $W$ acts co-compactly on $X$, there is an integer $K_1$ such that every point of $X$ is within $K_1$ of the orbit $Wx$. For each integer $0, 1,\ldots ,N$ such that $N$ is less that or equal to the length of $\alpha$, choose a point $v_ix$ of $Wx$ within $K_1$ of $\alpha (i)$. Let $\beta_i$ be a $\Lambda$-geodesic connecting $v_i$ to $v_{i+1}$ and $\beta$ be the $\Lambda$-edge path $(\beta_0,\beta_1,\ldots)$. Since the map $P_x:\Lambda \to \Lambda_x$ is quasi-isometric, there are numbers $\lambda$ and $\epsilon$ such that any such $\beta$ is a $(\lambda, \epsilon)$-quasi-geodesic in $\Lambda$ and  numbers $D_{\Lambda}$ and $D_X$ such that the length of any $\beta_i$ is less than or equal to $D_{\Lambda}$ (in $\Lambda$) and every point of such a $P_x(\beta_i)$ is within $D_X$ of $\alpha(i)$ (in $X$). Certainly every point of $\alpha$ is within $K\equiv K_1+1$ of $P_x(\beta)$.  

\medskip


\hspace{.5in}\includegraphics[scale=.9]{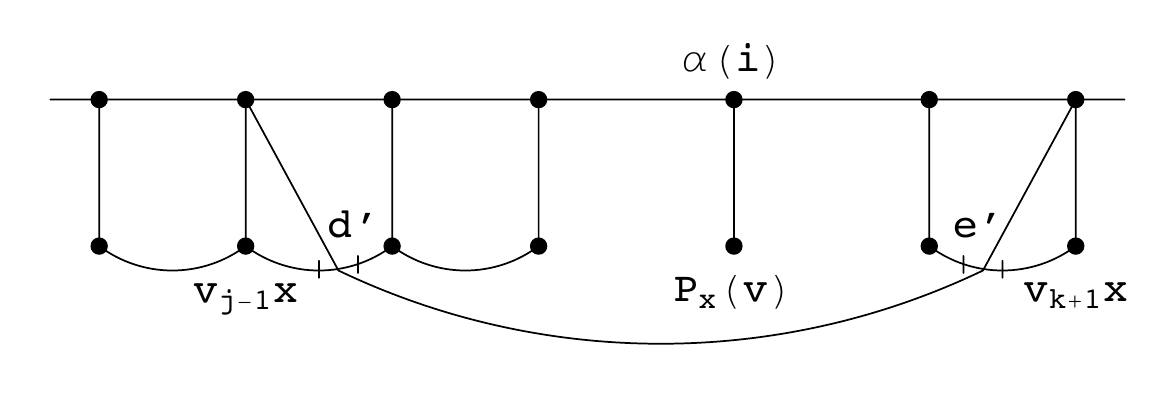}

\centerline{Figure 5}

\medskip

\medskip


Hence it suffices to bound the bracket number of such a $\beta$. If $v$ is a vertex of $\beta_i$ and $e$ and $d$ are edges of $\beta$ preceding and following $v$ respectively such that $e$ and $d$ belong to the same wall $Q$ of $\Lambda$, then $e$ is an edge of $\beta_j$ and $d$ is an edge of $\beta_k$ where $j\leq i\leq k$. The mid-points $e'$ of $P_x(e)$ and $d'$ of $P_x(d)$ are fixed (in  $\Lambda_x$ and $X$) by the reflection $r_Q\in W$ for the wall $Q$ (as are the mid-points of $e$ and $d$ in $\Lambda$). 
Hence the geodesic in $X$ connecting $d'$ and $e'$ is fixed by $r_Q$. Now, $d'$ (respectively $e'$) is within $D_X$ of $\alpha(j-1)$ (respectively $\alpha(k+1)$) and $P_x(v)$ is within $D_X$ of $\alpha(i)$. By the CAT(0) inequality for quadrilaterals (in particular for the quadrilateral determined by $d'$, $e'$, $\alpha (j-1)$, and $\alpha (k+1)$) $\alpha(i)$ is within $D_X$ of a point of the $X$-geodesic connecting $d'$ to $e'$ and hence $\alpha(i)$ is within $D_X$ of a fixed point of $r_Q$.  (See figure 5.)

Since the action of $W$ on $X$ is properly discontinuous, there is a bound $B$ on the number of reflections $r_Q$ such that $r_Q$ does not take the ball of radius $D_X$ centered at $v(i)\in X$ (equivalently centered at any $x\in X$) off of itself. Hence there cannot be more than $B$ walls bracketing the vertex $v$ of $\beta$. 
\end{proof}

\begin{remark}
Note that the above proof is valid even when $W$ does not act co-compactly on the CAT(0) space $X$, as long as the CAT(0) geodesic remains a bounded distance from $\Lambda_x$ for some $x$. 
\end{remark}

The following result answers a question posed by K. Ruane.

\begin{corollary}
Suppose $(W,S)$ is a finitely generated Coxeter group acting geometrically on the CAT(0) space $X$. For $x\in X$ let $\Lambda_x$ be the copy of the Cayley graph of $(W,S)$ in $X$, (with $W$-equivariant map $P_x:\Lambda(W,S)\to \Lambda_x$). Then for each subset $A\subset S$, the subgroup $\langle A\rangle$ is quasi-convex in $X$.  (I.e. $P_x(\langle A\rangle)$ is quasi-convex in $X$.)
\end{corollary}
\begin{proof}
Let $K$ be the tracking constant from corollary \ref{CAT(0)}. Suppose $a_1, a_2\in A$ and $\alpha$ is a CAT(0) geodesic in $X$ from $P_x(a_1)$ to $P_x(a_2)$. Let $\beta$ be a $\Lambda_x$, edge path geodesic which $K$-tracks $\alpha$. I.e. there is a  $\Lambda(W, S)$ geodesic $\beta'$, from $a_1$ to $a_2$ such that $P_x(\beta')=\beta$. Since $a_i\in A$, the edge labels of $\beta'$ are all in $A$. This means all vertices of $\beta'$ are in $\langle A\rangle$,  and so the image of $\alpha$ is within $K$ of $P_x(\langle A\rangle)$. 
\end{proof}

The next result says that elements of infinite order in a Coxeter group are tracked by geodesics in the standard Cayley graph. 

\begin{corollary} Suppose $(W,S)$ is a finitely generated Coxeter system and $g\in W$ is an element of infinite order. Then in the Cayley graph $\Lambda(W,S)$ the elements $\{\ldots, g^{-2}, g^{-1}, 1, g, g^2,\ldots \}$ are tracked by a Cayley graph geodesic.
\end{corollary}

\begin{proof} By G. Moussong \cite{Moussong}, all finitely generated Coxeter groups are CAT(0). Let $X$ be any CAT(0) space such that $W$ acts geometrically on $X$. The min set of $g$ contains a geodesic line $l$ that is invariant under the action of $g$. Let $x$ be any point in $X$ and $\Lambda_x$ the copy of $\Lambda(W,S)$ in $X$ at $x$. Let $\alpha$ be an $S$-geodesic for $g$. Observe that the edge path line  $l_g$ in $\Lambda _x$ determined by positive and negative iterates of $\alpha$ at $x$ is a bounded distance from $l$. The proof of corollary \ref{CAT(0)} shows that $l_g$  is a quasi-geodesic with bounded bracket number and so by corollary \ref{localfinite} is tracked by a Cayley graph geodesic. 
\end{proof} 

One of the fundamental asymptotic results for word hyperbolic groups is that 1-ended word hyperbolic groups have locally connected boundary. This result follows from a long program of results by several authors, notably B. Bowditch, and concluded by G. Swarup \cite{Swarup}. To give a feeling for the reach of our results, we outline an elementary proof of this fact for Coxeter groups. 
\begin{corollary}
If $W$ is a 1-ended word hyperbolic Coxeter group then the boundary of $W$ is locally connected.
\end{corollary}
\begin{proof}
We use an elementary form of a construction of a ``filter" in \cite{MRT} (where a partial classification of right angled Coxeter groups with locally connected boundaries is produced). Suppose $W$ acts geometrically on the CAT(0) space $X$, with base point $x$. Let $\Lambda_x$ be the copy of the Cayley graph of $(W,S)$ at $x$ in $X$ with proper $W$-equivariant map $P_x:\Lambda(W,S)\to \Lambda_x$. Suppose $r$ and $s$ are ``close" geodesic rays in $X$, with $r(0)=s(0)=x$. Choose $\Lambda$ (edge path) geodesics $r'$ and $s'$ at $\ast$ (the identity vertex of $\Lambda(W,S)$), such that $P_x(r')$ and $P_x(s')$ $K$-track $r$ and $s$ respectively. Since $r$ and $s$ are close in $\partial X$, we may assume that $r'$ and $s'$ have long initial segments with ``close" terminal points. For simplicity we assume these initial segments agree. If $y$ is the last vertex of this common initial segment, say the edge of $r'$ following $y$ has label $a_1$ and the edge of $s'$ following $y$ has lablel $b_1$. 
The presentation diagram $\Gamma(W,S)$ of $(W,S)$ has vertex set $S$ and an edge labeled $m(i,j)$ between distinct vertices $s_i, s_j$ if $m(i,j)\not = \infty$. Since $W$ is 1-ended no subset $A$ of $S$ with $\langle A\rangle$ a finite group separates $\Gamma$ (see corollary 16 of \cite{MT}).
 The set $B$ of $S$-elements that label edges at $y$ with end points closer to $\ast $ than $y$ is to $\ast$ generates a finite subgroup of $W$. 
The set of vertices of $\Gamma$ corresponding to $B$ does not separate $\Gamma$ and $B$ does not contain $a_1$ or $b_1$.  Hence there is an edge path in $\Gamma$ from $a_1$ to $b_1$ avoiding $B$. Let the consecutive vertices of this path be $a_1=v_1,v_2,\ldots , v_n=b_1$. 
If $q(i,i+1)$ is the (finite) order of $v_iv_{i+1}$ then the relation $(v_i,v_{i+1})^{q(i,i+1)}$ determines a loop at $y\in \Lambda$ such that the two half loops at $y$ making up this loop extend the Cayley geodesic from $\ast$ to $y$. 
Consider the subgraph $F_1$ of $\Lambda$ determined by the edge paths $r'$, $s'$ and the edge loops for each $v_iv_{i+1}$. Each $v_i$ determines an edge of $F_1$ (with label $v_i$) beginning at $y$. At the end point of this edge there are two edges of $F_1$ that extend a Cayley geodesic from $\ast$ to $y$. 
Build a set of loops as with $a_1$ 
and $b_1$ for each of these pairs of edges. Then $F_2$ is $F_1$ union all new loops. Continuing we build a 1-ended subgraph $F=\cup_{i=1}^{\infty}F_i$ of $\Lambda$ such that for each vertex $v$ of $F$, not on the common overlap of $r'$ and $s'$, there is a Cayley geodesic from $\ast$ to $v$ in $F$ which passes through $y$. We claim that $L$, the limit set 
of $P_x(F)$ is a ``small" connected set containing $r$ and $s$ (and so $\partial X$ is locally connected). Certainly, $r$ and $s$ are in $L$. Since $F$ is 1-ended and $P_x$ is proper, $L$ is connected. If $v$ is a vertex of $F$, then there is a Cayley geodesic $\alpha_v$ from $
\ast$ to $v$ (which passes through $y$ for all but finitely many $v$). If $z\in L$ then let $z_1,z_2,\ldots $ be a sequence of vertices of $F$ such that $P_x(z_i)$ converges to $z$. 
The CAT(0) geodesic from $x$ to $P_x(z_i)$ is $K$- tracked by a Cayley geodesic $\beta_i$ in $\Lambda_x$. As $W$ is word hyperbolic the Cayley geodesics $P_x(\alpha_{v_i})$ and $\beta_i$ (with the same end points)
must $\delta$-fellow travel (for a fixed constant $\delta$). In particular each $\beta_i$ must pass ``close" to $P_x(y)$ and so $z$ is close to both $r$ and $s$ in $\partial X\equiv \partial W$.
\end{proof}


\begin{thebibliography}{MR}
\bibitem[BH]{BridsonHaefliger}
M.R. Bridson, A. Haefliger. {\em Metric Spaces of Non-positive Curvature}
(Grundl. Math. Wiss., Vol. 319, Springer,  Berlin 1999).

\bibitem[BrH]{BrinkHowlett} B. Brink, R.B. Howlett. A finiteness property and an automatic structure for Coxeter groups. {\em Math.
Ann.} {\bf 296(1)} (1993), 179-190.

\bibitem[D]{Davis}
M.W. Davis, {\em The geometry and topology of Coxeter groups.} London Mathematical Society Monographs Series 
Vol. 32, Princeton University Press, Princeton, NJ, 2008. 

\bibitem[Di]{Dilworth}  R.P. Dilworth. A Decomposition Theorem for Partially Ordered Sets. {\em Ann.  of Math.} {\bf 51} (1950), 161-166.


\bibitem[MRT]{MRT}
M. Mihalik, K. Ruane, S. Tschantz. Local connectivity of right angled Coxeter group boundaries. {\em J. Group Theory} {\bf 10} (2007), 531-560.

\bibitem[MT]{MT} M. Mihalik, S. Tschantz. Visual decompositions of Coxeter grouip. {\em Groups Geom. Dyn.} {\bf 3} (2009), 173-198.

\bibitem[Mo]{Moussong}
G. Moussong.  Hyperbolic Coxeter groups.  PhD. thesis. Ohio
State University (1988).


\bibitem[Sh]{Short}
H. Short (ed.) Notes on word hyperbolic groups. In {\em Group Theory from a Geometrical Viewpoint} (E. Ghys, A. Haefliger and A. Verjovsky ed.)
World Scientific 1990, 3-64. 

\bibitem[S]{Swarup}
G. Swarup.  On the cut point conjecture.  {\em Electron. Res.
Announc. Amer. Math. Soc.} {\bf 2} (1996).

\end{thebibliography}
\enddocument